\documentclass[11pt,reqno]{amsart}

\usepackage{amsmath,amssymb,amsthm,esint}
\usepackage{bm}
\usepackage{mathrsfs}
\usepackage[all]{xy}
\usepackage{booktabs}
\usepackage{fullpage}
\usepackage[pagebackref=true]{hyperref}
\usepackage{mathtools}
\usepackage[abbrev,msc-links]{amsrefs}
\usepackage[]{todonotes}
\hypersetup{
 hidelinks,
 bookmarksopen=true,
}

\newcommand{\nnb}{\mathbb{N}}
\newcommand{\rl}{\mathbb{R}}
\newcommand{\cx}{\mathbb{C}}

\newcommand{\CB}{\mathcal{B}}

\newcommand{\CN}{\mathcal{N}}

\newcommand{\ai}{\sqrt{-1}}

\newcommand{\inj}{\hookrightarrow}

\newcommand{\prj}{\mathbb{P}}

\theoremstyle{plain}
\newtheorem{theorem}{Theorem}[section]

\newtheorem{lemma}[theorem]{Lemma}

\theoremstyle{definition}

\theoremstyle{definition}
\newtheorem{remark}[theorem]{Remark}

\newtheorem{problem}[theorem]{Problem}

\begin{document}

\title{Quantitative injectivity of the Fubini--Study map}
\author[Y.~Hashimoto]{Yoshinori Hashimoto}
\date{\today}
\address{Department of Mathematics, Osaka Metropolitan University, 3-3-138, Sugimoto, Sumiyoshi-ku, Osaka, 558-8585, Japan.}
\email{yhashimoto@omu.ac.jp}

\begin{abstract}
We prove a quantitative version of the injectivity of the Fubini--Study map that is polynomial in the exponent of the ample line bundle, and correct the arguments in the author's previous papers \cite{yhhilb,yhextremal}.
\end{abstract}


\maketitle

\tableofcontents

\section{Introduction}

Suppose that we have a polarised smooth projective variety $(X , L)$ over $\cx$ of complex dimension $n$ and a projective embedding $\iota_k : X \inj \prj (H^0(X , L^k)^{\vee} )$, where $L$ is an ample line bundle over $X$ and $k$ is a large enough integers such that $L^k$ is very ample. We may moreover assume, by replacing $L$ by a higher tensor power if necessary, that the automorphism group $\mathrm{Aut}_0 (X,L)$ acts on $(X,L)$ via the ambient linear action of $SL(H^0(X , L^k))$ fixing $\iota_k (X)$. We fix a hermitian metric $h$ on $L$ which defines a K\"ahler metric $\omega_h >0$ on $X$, where $\omega_h \in c_1(L)$. This in turn defines a positive definite hermitian form $\mathrm{Hilb}(h^k)$ on $H^0(X , L^k)$ defined (see \cite{donproj2}) as
\begin{equation*}
	\mathrm{Hilb}(h^k)(s_1 , s_2) := \frac{N_k}{V} \int_X h^k (s_1,s_2) \frac{\omega^n_h}{n!}
\end{equation*}
for $s_1 , s_2 \in H^0(X , L^k)$, where we wrote
\begin{equation*}
	N_k := \dim_{\cx} H^0(X , L^k), \quad V:= \int_X c_1(L)^n/n!.
\end{equation*}
Choosing a $\mathrm{Hilb}(h^k)$-orthonormal basis $\{ s_i \}_{i=1}^{N_k}$, we identify $\prj (H^0(X , L^k)^{\vee} )$ with $\prj^{N_k-1}$, and also identify the set $\CB_k$ of all positive definite hermitian forms on $H^0(X , L^k)$ with positive definite $N_k \times N_k$ hermitian matrices, noting that $\mathrm{Hilb}(h^k)$ is the identity matrix with respect to this basis. The set $\CB_k$ in turn can be identified with the symmetric space $GL(N_k, \cx) / U(N_k)$. The Fubini--Study map $\mathrm{FS}_k $ associates to $A \in \CB_k$ a hermitian metric $\mathrm{FS}_k (A)$, called the Fubini--Study metric, on $L^k$. Recall that $\mathrm{FS}_k (A)$ is defined as a unique hermitian metric on $L^k$ which satisfies
\begin{equation} \label{eqdfs}
	\sum_{i=1}^{N_k} \mathrm{FS}_k (A) (s^A_i, s^A_i) = 1
\end{equation}
over $X$, for an $A$-orthonormal basis $\{ s^A_i \}_{i=1}^{N_k}$ for $H^0(X , L^k)$. When we write $H_k := \mathrm{Hilb}(h^k)$, the equation above immediately implies
\begin{equation*}
	\mathrm{FS}_k (A) = \left( \sum_{i,j=1}^{N_k} A^{-1}_{ij} \mathrm{FS}_k (H_k) (s_i , s_j) \right)^{-1} \mathrm{FS}_k (H_k) ,
\end{equation*}
by noting that $\{ \sum_j A^{-1/2}_{ij} s_j \}_{i=1}^{N_k}$ is an $A$-orthonormal basis. For $A,B \in \CB_k$, we define a smooth real function $f_k (A,B;H_k)$ on $X$ defined as
\begin{align*}
	f_k (A,B;H_k) &:= \frac{\mathrm{FS}_k (H_k)}{\mathrm{FS}_k (A)}  - \frac{\mathrm{FS}_k(H_k)}{\mathrm{FS}_k (B)} \\
	&=\sum_{i,j=1}^{N_k} (A^{-1}_{ij} - B^{-1}_{ij}) \mathrm{FS}_k (H_k) (s_i , s_j).
\end{align*}
Given $A, B \in \CB_k$, it is natural to conjecture that $A$ is close to $B$ in $\CB_k$ if and only if $f_k (A,B;H_k)$ is close to zero in some appropriate norm. Indeed, Lempert \cite[Theorem 1.1]{Lem21} proved that the Fubini--Study map is injective, which is equivalent to saying that $A=B$ if and only if $f_k(A,B;H_k) \equiv 0$, and also proved that its image is not closed \cite[Theorem 1.1, Example 6.2, Theorem 6.4]{Lem21}. In this paper, we consider the following quantitative version of injectivity.

\begin{problem} \label{prqinjfs}
	Can we bound an appropriately defined distance between $A$ and $B$ in $\CB_k$ by $C k^l \Vert f_k (A,B;H_k) \Vert$, where $C,l>0$ are some constants which depend only on $h$? 
\end{problem}

The problem here is that the constant appearing in the estimate is at most polynomial in $k$, as well as the right choice of the norm for $f_k (A,B;H_k)$. In \cite[Lemma 3.1]{yhhilb}, the author claimed that this problem can be solved with the $C^0$-norm, but its proof does not hold since it is based on an erroneous claim that the Hilbert map is surjective \cite[Theorem 1.1, Proposition 2.15]{yhhilb}. Indeed, a counterexample was provided by Lempert in an email correspondence with the author (see \cite[Proposition 4.16]{Finski22} and also \cite[Proposition 3.6]{Finski22s}), and there is also a counterexample to the surjectivity of the Hilbert map by Sun \cite[\S 3, Example]{Sun22}. Both these counterexamples involve basis elements that are redundant in terms of the global generation of the line bundle (see the definition of a rather ample subspace in \cite[\S 6]{Lem21}). Sun \cite[Theorems 1.2 and 1.3]{Sun22} also described the closure of the image of the Hilbert map. 

It is still plausible, however, that Problem \ref{prqinjfs} can be solved affirmatively once the right norm is chosen. The main result of this paper is the following.

\begin{theorem} \label{ppqifs}
	There exist constants $k_0 \in \mathbb{N}$ and $C_h >0$, depending only on $h$, such that we have
\begin{equation*}
		\Vert A^{-1} - B^{-1} \Vert^2_{\mathrm{HS}(H_k)} \le C_h k^{n} \Vert f_k (A, B ;H_k) \Vert^2_{W^{2,2}(\omega_h)} 
\end{equation*}
for any $A, B \in \CB_k$ and for any $k \ge k_0$, where $\Vert \cdot \Vert_{\mathrm{HS}(H_k)}$ is the Hilbert--Schmidt norm with respect to $H_k = \mathrm{Hilb}(h^k)$ and $\Vert \cdot \Vert_{W^{2,2} (\omega_h)}$ is the Sobolev norm with respect to $\omega_h$.
\end{theorem}

We stress here that $A,B \in \CB_k$ are regarded as matrices with respect to the $H_k$-orthonormal basis, when we consider the inverse matrices above. 

The choice of the reference metric $H_k$ is important in Theorem \ref{ppqifs}; see (\ref{eqdfrfmc}) for the formula with respect to a different reference. Indeed, the counterexamples of Lempert and Sun mentioned above deal with the case when the reference metric is close to being degenerate. 

\medskip

\noindent \textbf{Acknowledgements.}  The author thanks Siarhei Finski, L{\'a}szl{\'o} Lempert, and Jingzhou Sun for very helpful discussions and immensely valuable examples. This research is partially supported by JSPS KAKENHI (Grant-in-Aid for Scientific Research (C)), Grant Number JP23K03120, and JSPS KAKENHI (Grant-in-Aid for Scientific Research (B)) Grant Number JP24K00524.

\section{Proof of Theorem \ref{ppqifs}}

In what follows, we write $\omega_{h,k} \in k c_1 (L)$ for the K\"ahler metric on $X$ associated to $\mathrm{FS} (H_k)$, and $\tilde{\omega}_{h,k}$ for the Fubini--Study metric on $\mathcal{O}_{\mathbb{P}^{N_k-1}} (1)$ over $\mathbb{P}^{N_k-1} = \prj (H^0(X , L^k)^{\vee} )$ defined by $H_k$. Recall $\omega_{h,k} = \iota^*_k \tilde{\omega}_{h,k}$, and that we have 
\begin{equation*}
	\omega_{h,k} = k \omega _h - \frac{\ai}{2 \pi} \partial \bar{\partial} \log \rho_k (\omega_h ) 
\end{equation*}
by a result due to Rawnsley \cite{rawnsley}, where we wrote $\rho_k (\omega_h)$ for the $k$-th Bergman function with respect to $\omega_h$. We also use the re-scaled version $\bar{\rho}_k (\omega_h ) := \frac{V}{N_k} \rho_k (\omega_h)$. We recall the asymptotic expansion of the Bergman function (see e.g.~\cite{mm})
\begin{equation} \label{eqbgexp}
	\rho_k (\omega_h) = k^n + O(k^{n-1}), \quad \text{or} \quad \bar{\rho}_k (\omega_h )  = 1 + O(1/k),
\end{equation}
where $N_k = V k^n + O(k^{n-1})$ by the asymptotic Riemann--Roch theorem. We also recall that the expansion above holds in $C^m$-norm for any $m \in \nnb$. We thus have
\begin{equation*}
	\omega_{h,k} = k \omega_h + O(1/k).
\end{equation*}

We start with the following lemma concerning the second fundamental form of the projective embedding.

\begin{lemma} \label{lmndsff}
	Suppose $H_k = \mathrm{Hilb} (h^k)$, and let $A_{h,k}$ be the second fundamental form of $TX$ in $\iota^* T \mathbb{P}^{N_k-1}$ with respect to $\tilde{\omega}_{h,k}$. Then, there exist $k_0 \in \nnb$ and a constant $\lambda_{\mathrm{min}} (h) = \lambda_{\mathrm{min}} >0$ depending only on $h$, such that the minimum eigenvalue of $-A_{h,k}^* \wedge A_{h,k}$ with respect to a $g_{h,k}$-orthonormal basis can be bounded below by $\lambda_{\mathrm{min}} (h)$ for all $k \ge k_0$ and all $x \in X$. In particular, $A_{h,k}$ is nondegenerate everywhere on $X$ for all large enough $k$.
\end{lemma}

\begin{proof}
	We recall that the curvature of the Fubini--Study metrics on the ambient projective space and $X$ can be related by
	\begin{equation*}
		-A_{h,k}^* \wedge A_{h,k} = \pi_T \circ (\tilde{F}_{h,k} |_{TX}) - F_{h,k}
	\end{equation*}
	as given in \cite[page 78]{grihar} or \cite[equation (5.27)]{PS04}, where we wrote $\tilde{F}_{h,k}$ for the curvature form of $\tilde{\omega}_{h,k}$ on the ambient projective space and $F_{h,k}$ for that of $\omega_{h,k}$ on $X$. The adjoint above is with respect to $\tilde{\omega}_{h,k}$. We recall that the curvature tensor of the Fubini--Study metric on $\mathbb{P}^{N_k-1}$ can be written as
	\begin{equation*}
		(\tilde{F}_{h,k})_{i \bar{j} l \bar{m}} = (\tilde{g}_{h,k})_{i \bar{j}} (\tilde{g}_{h,k})_{l \bar{m}} + (\tilde{g}_{h,k})_{i \bar{m}} (\tilde{g}_{h,k})_{l \bar{j}}
	\end{equation*}
	where $\tilde{g}_{h,k}$ is the metric tensor of $\tilde{\omega}_{h,k}$ written with respect to its orthonormal basis.
	
	By the Bergman kernel expansion (\ref{eqbgexp}), we find that the restriction of $\tilde{g}_{h,k}$ to $TX$ satisfies the asymptotic expansion
	\begin{equation*}
		g_{h,k} = kg_h + O(1/k),
	\end{equation*}
	where $g_h$ stands for the K\"ahler metric associated to $h$. Thus we have
	\begin{equation*}
		{{(\pi_T \circ (\tilde{F}_{h,k} |_{TX}) - F_{h,k})_{i}}^{j}}_{l \bar{m}} = \delta_{i}^j \delta_{l \bar{m}} + \delta_{l}^j \delta_{i \bar{m}} - {{(F_h)_{i}}^{j}}_{l \bar{m}} + O(1/k)
	\end{equation*}
	with respect to any $kg_h$-orthonormal basis for $TX$ (note that the curvature tensor $F_h$ is of order 1 with respect to a $kg_h$-orthonormal basis, which is of order $k^{-1}$ with respect to a $g_h$-orthonormal basis). Since this agrees with $-A_{h,k}^* \wedge A_{h,k}$, the tensor above is positive semidefinite.
	
	Suppose that $-A_{h,k}^* \wedge A_{h,k}$ is degenerate at $x \in X$. We then perturb $h$ locally around $x$ so that the perturbation $h_{ \epsilon}$ satisfies
	\begin{equation*}
		h (x) = h_{ \epsilon} (x), \quad g (x) = g_{\epsilon} (x)
	\end{equation*}
	and the curvature tensor is perturbed by $- \epsilon \delta_i^j (k g_h)_{l \bar{m}}$ at $x$, which is of order $k \epsilon$ locally at $x$, and such that $\mathrm{Hilb} (h_{\epsilon}^k) = \mathrm{Hilb} (h^k) + O(\epsilon^2)$. This is possible by considering the perturbation of the K\"ahler potential by $\epsilon |z_i|^2 |z_l|^2$ where $(z_1 , \dots , z_n)$ are holomorphic normal coordinates around $x$, multiplied by a cutoff function which decreases very rapidly away from $x$ so that the perturbation introduced in the integral inside $\mathrm{Hilb}$ is small.

	Applying the argument above to $h_{\epsilon}$ and writing with respect to a $g_h$-orthonormal basis, we find
	\begin{align*}
		-A_{h,k, \epsilon}^* \wedge A_{h,k, \epsilon} = \pi_T \circ (\tilde{F}_{h,k,\epsilon} |_{TX}) - F_{h,k,\epsilon} &= -A_{h,k}^* \wedge A_{h,k} - \epsilon \delta_i^j (k g_h)_{l \bar{m}} + O(\epsilon^2) + O(1/k) \\
		&=-A_{h,k}^* \wedge A_{h,k} - \epsilon \delta_i^j \delta_{l \bar{m}} + O(\epsilon^2) + O(1/k),
	\end{align*}
	where $\tilde{F}_{h,k,\epsilon}$ is the curvature of the Fubini--Study metric on $\prj^{N_k-1}$ with respect to $\mathrm{Hilb}(h^k_{\epsilon})$, $F_{h,k,\epsilon}$ is the one with respect to its restriction to $TX$, and $A_{h,k, \epsilon}$ is the second fundamental form that is associated to it. This yields a contradiction for a sufficiently small $\epsilon >0$ and for all large enough $k$, if we assume that $-A_{h,k}^* \wedge A_{h,k}$ is degenerate at $x \in X$.
	
	The second claim is a consequence of the fact that, with respect to a $kg_h$-orthonormal basis, $\pi_T \circ (\tilde{F}_{h,k} |_{TX}) - F_{h,k}$ converges to a positive definite form on $X$ by the previous argument and the Bergman kernel expansion $g_{h,k} = kg_h + O(1/k)$, together with the compactness of $X$.
\end{proof}

Given $A, B \in \CB_k$, we define a (not necessarily positive) hermitian matrix $\Lambda$ by
\begin{equation*}
	\Lambda := A^{-1} - B^{-1},
\end{equation*}
and also define
\begin{equation*}
	f(\Lambda ; H_k) := f(A, B;H_k).
\end{equation*}
We write $\xi_{\Lambda}$ for the holomorphic vector field on $\prj^{N_k -1}$ generated by the linear action of $\Lambda$. We recall that, when $[Z_1 : \cdots : Z_{N_k}]$ is a set of homogeneous coordinates for $\mathbb{P}^{N_k -1}$ consisting of  $H_k$-orthonormal basis for $H^0 (X, L^k)$, the Hamiltonian function for $\xi_{\Lambda}$ with respect to $\tilde{\omega}_{h,k}$ is given by
\begin{equation*}
	\tilde{f} (\Lambda ; H_k) := \sum_{i,j=1}^{N_k} \Lambda_{ij} \frac{Z_i \bar{Z}_j}{\sum_{l=1}^{N_k} |Z_l|^2},
\end{equation*}
which is a well-defined smooth function on $\mathbb{P}^{N_k -1}$. By the definition of $\mathrm{FS}_k (H_k)$, we have
\begin{equation*} \label{eqhamfspb}
	f (\Lambda; H_k) = \iota^*_k \tilde{f} (\Lambda; H_k).
\end{equation*}

We first consider the case when the trace of $\Lambda$ is zero, i.e.~when $\Lambda \in \mathfrak{sl}(N_k, \cx)$. Following \cite[\S 5]{PS04}, we consider a holomorphic vector bundle $\iota^* T \mathbb{P}^{N_k-1}$ over $X$, which we decompose as
\begin{equation} \label{eqdctn}
	\iota^* T \mathbb{P}^{N_k-1} = TX \oplus \mathcal{N},
\end{equation}
where the direct sum is orthogonal with respect to the Fubini--Study metric on $\iota^* T \mathbb{P}^{N_k-1}$ with respect to $H_k$. We write $\pi_T$ and $\pi_{\CN}$ for the projection to $TX$ and $\CN$. By the defining equation
\begin{equation*}
	\iota (\xi_{\Lambda}) \tilde{\omega}_{h,k} = -d \tilde{f} (\Lambda ; H_k)
\end{equation*}
for the Hamiltonian function for $\xi_{\Lambda}$, together with $\omega_{h,k} = \iota^*_k \tilde{\omega}_{h,k}$ and (\ref{eqhamfspb}), we find that $\pi_T \xi_{\Lambda}$ is $\omega_{h,k}$-metric dual to $ d f_k ( \Lambda ;H_k)$ as pointed out in \cite[Lemma 20]{Fine10}.

Recall that we have a faithful representation $\mathrm{Aut}_0 (X,L) \inj SL (N_k, \cx)$ by the linearisation of the action of $\mathrm{Aut}_0 (X,L)$, which is unique up to an overall multiplicative constant (by replacing $L$ by a higher tensor power if necessary). We further decompose $\Lambda = \alpha + \beta$ according to $\mathfrak{sl}(N_k, \cx) = \mathfrak{aut}(X,L) \oplus \mathfrak{aut}(X,L)^{\perp}$, where the orthogonality is defined with respect to the $L^2$-metric defined by $\omega_{h,k}$ which is the Fubini--Study metric on $\iota^* T \mathbb{P}^{N_k-1}$ with respect to $H_k$. We then have
\begin{equation*}
	\Vert \xi_{\Lambda} \Vert^2_{L^2 (\omega_{h,k})} = \Vert \xi_{\alpha+ \beta} \Vert^2_{L^2 (\omega_{h,k})} = \Vert \xi_{\alpha} \Vert^2_{L^2 (\omega_{h,k})} + \Vert \xi_{\beta} \Vert^2_{L^2 (\omega_{h,k})}
\end{equation*}
by definition. Note also that, according to the orthogonal decomposition (\ref{eqdctn}), we have
\begin{equation*}
	\Vert \xi_{\beta} \Vert^2_{L^2 (\omega_{h,k})} = \Vert \pi_T \xi_{\beta} \Vert^2_{L^2 (\omega_{h,k})} + \Vert \pi_{\CN} \xi_{\beta} \Vert^2_{L^2 (\omega_{h,k})}
\end{equation*}
because of the pointwise orthogonality, and also note $\xi_{\alpha} = \pi_T \xi_{\alpha}$. 

\begin{lemma} \label{lmpnxbdb}
	Suppose $\iota_k : X \inj \mathbb{P} (H^0 (X, L^k)^{\vee})$, and that we consider the Fubini--Study metric $\mathrm{FS} (H_k)$ where $H_k = \mathrm{Hilb} (h^k)$. Then there exists $C_1 >0$ depending only on $h$ such that
	\begin{equation*}
		\Vert \pi_{\CN} \xi_{\beta} \Vert^2_{L^2 (\omega_{h,k})} \le C_1 \Vert \bar{\partial} ( \pi_T \xi_{\beta}) \Vert^2_{L^2 (\omega_{h,k})}
	\end{equation*}
	holds for any $\beta \in \mathfrak{aut} (X,L)^{\perp}$ and for all large enough $k$.
\end{lemma}

\begin{proof}
	We prove
\begin{equation} \label{eqieqeqxb}
	\Vert \pi_{\CN} \xi_{\beta} \Vert^2_{L^2 (\omega_{h,k})} \le C_1 \Vert \bar{\partial} ( \pi_{\CN} \xi_{\beta} )\Vert^2_{L^2 (\omega_{h,k})} = C_1 \Vert \bar{\partial} ( \pi_T \xi_{\beta} ) \Vert^2_{L^2 (\omega_{h,k})} .
\end{equation}

The equality in (\ref{eqieqeqxb}) is straightforward, following \cite[page 705]{PS04}, since
\begin{equation*}
	\bar{\partial}_{\mathbb{P}^{N_k-1}} \xi_{\beta} = 0 = \bar{\partial}_{\mathbb{P}^{N_k-1}} ( \pi_{\CN} \xi_{\beta}) + \bar{\partial}_{\mathbb{P}^{N_k-1}} (\pi_T \xi_{\beta}),.
\end{equation*}
which in turn implies
\begin{equation*}
	\bar{\partial} \xi_{\beta} = 0 = \bar{\partial} ( \pi_{\CN} \xi_{\beta} ) + \bar{\partial} ( \pi_T \xi_{\beta} ),
\end{equation*}
where $\bar{\partial}$ makes sense as a $\bar{\partial}$-operator on the holomorphic vector bundle $\iota^* T \mathbb{P}^{N_k-1}$.

Thus it suffices to prove the inequality in (\ref{eqieqeqxb}), which is the reverse direction of \cite[equation (5.19)]{PS04}, and we prove it by using the lower bound of the second fundamental form given in Lemma \ref{lmndsff}.


We reproduce parts of the argument in \cite[\S 5]{PS04} here for completeness. Fix a point $x \in X$ and a local holomorphic frame $\{ e_1 , \dots , e_n , f_1 , \dots , f_m \}$ of $\iota^* T \mathbb{P}^{N_k-1}$ (with $N-1 = n+m$) in a neighbourhood of $x$ such that it is an orthonormal basis for $\iota^* T_x \mathbb{P}^{N_k-1}$ with respect to $\tilde{g}_{h,k}$ and $\{ e_1 , \dots , e_n \}$ form a local holomorphic frame of $TX$ near $x$. We then write
\begin{equation*}
	\xi_{\beta} = \sum_{i=1}^n a_i e_i + \sum_{j=1}^m b_j f_j
\end{equation*}
where $a_1 , \dots , a_n , b_1 , \dots , b_m$ are (local) holomorphic functions. We then have
\begin{equation*}
	\pi_{\CN} \xi_{\beta} = \sum_{j=1}^m b_j \left( f_j -\sum_{i=1}^n \phi_{ij} e_i \right)
\end{equation*}
for some smooth functions $\phi_{ij}$ vanishing at $x$. Then we have
\begin{equation*}
	\bar{\partial} ( \pi_{\CN} \xi_{\beta} ) = \sum_{j=1}^m b_j \left( - \sum_{i=1}^n (\bar{\partial} \phi_{ij} ) e_i .\right)
\end{equation*}

It is explained in \cite[\S 5]{PS04} that $\bar{\partial} \phi_{ij}$ is the second fundamental form of $TX$ with respect to $\tilde{\omega}_{k,h}$. By Lemma \ref{lmndsff}, there exists a constant $\lambda_{\mathrm{min}} (h) = \lambda_{\mathrm{min}} >0 >0$ depending only on $h$, such that
\begin{equation*}
	\sum_{i=1}^n \left| \sum_{j=1}^{N-1-n} b_j \bar{\partial} \phi_{ij} \right|^2 (x) \ge \lambda_{\mathrm{min}} \sum_{j=1}^{N-1-n} |b_j|^2 (x)
\end{equation*}
holds for any $x \in X$. Integrating both sides over $X$ with respect to the volume form $\omega^n_{h,k}$, we get the required estimate (\ref{eqieqeqxb}).
\end{proof}

\begin{proof}[Proof of Theorem \ref{ppqifs}]
We first recall that there exists $C_2 >0$ depending only on $h$ (and independent of $k$ as long as it is large enough)
\begin{equation*}
	\Vert \Lambda \Vert^2_{\mathrm{HS}(H_k)} \le C_2 k \Vert \xi_{\Lambda} \Vert^2_{L^2 (\omega_{h,k})} =  C_2 k \left( \Vert \xi_{\alpha} \Vert^2_{L^2 (\omega_{h,k})} + \Vert \xi_{\beta} \Vert^2_{L^2 (\omega_{h,k})} \right)
\end{equation*}
by \cite[(5.7)]{PS04}, where we note that the assumption (see \cite[(5.1)]{PS04}) for this estimate is satisfied for $\mathrm{FS}_k (H_k)$ since the associated centre of mass is the identity matrix up to an error of order $1/k$ (in the operator norm, after passing to a unitarily equivalent basis in which $D_k$ and $E_k$ are both diagonal) by the Bergman kernel expansion; note that $\Lambda$ is hermitian, and the Lie algebra $\mathfrak{su}(N)$ is isomorphic to the vector space consisting of hermitian matrices.

Since $\xi_{\alpha} = \pi_T \xi_{\alpha}$ is $L^2$-orthogonal to $\xi_{\beta}$, and pointwise orthogonal to $\pi_{\CN} \xi_{\beta}$, $\pi_T \xi_{\alpha}$ is $L^2$-orthogonal to $\pi_T \xi_{\beta}$. We thus find
\begin{equation*}
	\Vert \pi_T \xi_{\alpha} \Vert^2_{L^2 (\omega_{h,k})} + \Vert \pi_T \xi_{\beta} \Vert^2_{L^2 (\omega_{h,k})} = \Vert \pi_T \xi_{\alpha} + \pi_T \xi_{\beta} \Vert^2_{L^2 (\omega_{h,k})} = \Vert \pi_T \xi_{\alpha + \beta} \Vert^2_{L^2 (\omega_{h,k})}
\end{equation*}
by the linearity of $\pi_T$. Thus, combining these equalities we get
\begin{align*}
	\Vert \xi_{\alpha} \Vert^2_{L^2 (\omega_{h,k})} + \Vert \xi_{\beta} \Vert^2_{L^2 (\omega_{h,k})} &= \Vert \pi_T \xi_{\alpha} \Vert^2_{L^2 (\omega_{h,k})} + \Vert \pi_T \xi_{\beta} \Vert^2_{L^2 (\omega_{h,k})} + \Vert \pi_{\CN} \xi_{\beta} \Vert^2_{L^2 (\omega_{h,k})} \\
	&= \Vert \pi_T \xi_{\alpha + \beta} \Vert^2_{L^2 (\omega_{h,k})} + \Vert \pi_{\CN} \xi_{\beta} \Vert^2_{L^2 (\omega_{h,k})}.
\end{align*}
By Lemma \ref{lmpnxbdb}, we get
\begin{equation*}
	\Vert \Lambda \Vert^2_{\mathrm{HS}(H_k)} \le C_3 k \left( \Vert \pi_T \xi_{\Lambda} \Vert^2_{L^2 (\omega_{h,k})} + \Vert \bar{\partial} ( \pi_T \xi_{\Lambda} ) \Vert^2_{L^2 (\omega_{h,k})} \right),
\end{equation*}
for some $C_3 >0$ which depends only on $h$, where we used
\begin{equation*}
	\bar{\partial} ( \pi_T \xi_{\beta} ) = \bar{\partial} ( \xi_{\alpha} + \pi_T \xi_{\beta}) = \bar{\partial} ( \pi_T \xi_{\Lambda} ).
\end{equation*}
Recalling that $\pi_T \xi_{\Lambda}$ is $\omega_{h,k}$-metric dual to $ d f_k ( \Lambda ;H_k)$, we get
\begin{equation*}
	\left\Vert \pi_T \xi_{\Lambda} \right\Vert^2_{L^2 (\omega_{h,k})} = \left\Vert d f_k ( \Lambda ;H_k) \right\Vert^2_{L^2 (\omega_{h,k})} \le 2 k^{n-1} \left\Vert d f_k ( \Lambda ;H_k) \right\Vert^2_{L^2 (\omega_{h})}
\end{equation*}
for all large enough $k$, by noting $\omega_{h,k} = k \omega_h + O(1/k)$ which follows from the Bergman kernel expansion (\ref{eqbgexp}); we note that we used the natural dual metric  of $\omega_{h,k}$ on $T^*X$, which contributes to the factor $k^{-1}$ above. Similarly, we get
\begin{equation*}
	\left\Vert \bar{\partial} ( \pi_T \xi_{\Lambda} ) \right\Vert^2_{L^2 (\omega_{h,k})}  \le 2 k^{n-1} \left\Vert \nabla^{0,1} d f_k ( \Lambda ;H_k) \right\Vert^2_{L^2 (\omega_{h})},
\end{equation*}
where we wrote $\nabla$ is the covariant derivative of $\omega_h$ and $\nabla^{0,1} = \bar{\partial}$ is its $(0,1)$-part. Thus the estimates above give
\begin{equation*}
	\Vert \Lambda \Vert^2_{\mathrm{HS}(H_k)} \le 2 C_3 k^n \left( \Vert \nabla f_k ( \Lambda ;H_k) \Vert^2_{L^2 (\omega_h)} + \Vert \nabla \nabla f_k ( \Lambda ;H_k) \Vert^2_{L^2 (\omega_h)} \right),
\end{equation*}
when $\mathrm{tr} (\Lambda ) = 0$.

In general, we write $\Lambda = \Lambda_0 + cI$ where $c := \mathrm{tr} (\Lambda) / N_k$ and $\Lambda_0$ is the trace-free part of $\Lambda$. By diagonalising $\Lambda_0$ and recalling the definition of the Hilbert--Schmidt norm, we find
\begin{equation*}
	\Vert \Lambda \Vert^2_{\mathrm{HS}(H_k)} = \Vert \Lambda_0\Vert^2_{\mathrm{HS}(H_k)} + c^2 N_k .
\end{equation*}
We likewise decompose
\begin{equation*}
	f_k (\Lambda ;H_k) = f_k (\Lambda_0 ;H_k) + c \sum_{i=1}^{N_k} |s_i|^2_{\mathrm{FS}(H_k)} = f_k (\Lambda_0 ;H_k) + c
\end{equation*}
by recalling the definition (\ref{eqdfs}) of $\mathrm{FS}(H_k)$. Writing $\bar{\rho}_k (\omega_h ) = \frac{V}{N_k} \rho_k (\omega_h)$ and recalling $\mathrm{FS}_k (H_k) = \bar{\rho}_k (\omega_h)^{-1}h^k$ as in \cite{rawnsley}, the above equation gives
\begin{equation*}
	\bar{\rho}_k (\omega_h) f_k (\Lambda ;H_k) = \sum_{i=1}^{N_k} \lambda_i |s_i|^2_{h^k} +  c \bar{\rho}_k (\omega_h).
\end{equation*}
Integrating this over $X$, noting that $\{ s_i \}_{i=1}^{N_k}$ is a $\mathrm{Hilb}(h^k)$-orthonormal basis and that we have $\int_X \bar{\rho}_k (\omega_h) \omega^n_h /n! = V$, we get
\begin{equation*}
	cV = \int_X \bar{\rho}_k (\omega_h) f_k (\Lambda ;H_k) \frac{\omega_h^n}{n!} 
\end{equation*}
since $\mathrm{tr} (\Lambda_0)=0$, and hence
\begin{equation*}
	|c| = \frac{1}{V} \left| \int_X \bar{\rho}_k (\omega_h) f_k (\Lambda ;H_k) \frac{\omega_h^n}{n!} \right| \le \frac{1}{V} \Vert \bar{\rho}_k (\omega_h) \Vert_{L^2 (\omega_h)} \Vert f_k (\Lambda ;H_k) \Vert_{L^2 (\omega_h)} \le 2 \Vert f_k (\Lambda ;H_k) \Vert_{L^2 (\omega_h)},
\end{equation*}
by Cauchy--Schwarz and the Bergman kernel expansion (\ref{eqbgexp}).

Collecting all the estimates and setting $C_4:= 2 \max \{ C_3, 2 V \} >0$, which depends only on $h$ and not on $k$, we get
\begin{align*}
	\Vert \Lambda \Vert^2_{\mathrm{HS}(H_k)} &\le 2N_k \Vert f_k (\Lambda ;H_k) \Vert^2_{L^2(\omega_h)} + 2C_3 k^n \left( \Vert \nabla f_k ( \Lambda ;H_k) \Vert^2_{L^2 (\omega_{h})} + \Vert \nabla \nabla f_k ( \Lambda ;H_k) \Vert^2_{L^2 (\omega_{h})} \right) \\
	&\le C_4 k^{n} \Vert f_k (\Lambda ;H_k) \Vert^2_{W^{2,2}(\omega_h)}
\end{align*}
for all large enough $k$, since we have $N_k = Vk^n + O(k^{n-1})$ by the asymptotic Riemann--Roch theorem.
\end{proof}

\begin{remark}
	A heuristic interpretation of Theorem \ref{ppqifs} is as follows. The image of the Kodaira embedding $\iota_k : X \inj \prj (H^0(X , L)^{\vee} )$ is not contained in any proper linear subspace, and hence we expect that an appropriately defined norm of $\pi_T \xi_{\Lambda}$ should be equivalent to $\Vert \Lambda \Vert_{\mathrm{HS}(H_k)}$. The estimate depends on how ``degenerate'' $\iota_k (X)$ is in the projective space, i.e.~how close it is to being contained in a proper linear subspace. The argument above proves this intuition when we take the reference metric to be $\mathrm{FS}_k (H_k)$, for which the centre of mass of the embedding is in fact close to the identity matrix by the Bergman kernel expansion (\ref{eqbgexp}). Since $\iota_k$ is an algebraic morphism, we can also expect that the constant depends at most polynomially in $k$. This argument seems intuitively plausible, but no written proof seems to be available in the literature. 
\end{remark}

\section{Correction to \cite{yhextremal}}

There are several places in \cite{yhextremal} which requires corrections, since they depend on the results in \cite{yhhilb} that are reproduced in \cite[Lemmas 5 and 6]{yhextremal}. \cite[Lemma 5]{yhextremal} is not directly used in the rest of the argument in \cite{yhextremal}, and only affects the paper through \cite[Lemma 6]{yhextremal}. The full quantitative statement of \cite[Lemma 6]{yhextremal} is used only in \cite[page 3004 in \S 4.2, after (35)]{yhextremal}, and all the other places \cite[pages 2982 and 3002]{yhextremal} that need \cite[Lemma 6]{yhextremal} only need the injectivity of the Fubini--Study map which was also proved by Lempert \cite{Lem21}.

Thus it suffices to correct the argument in \cite[page 3004]{yhextremal}, which deals with the equation
\begin{equation} \label{eqextc}
	\sum_{i=1}^{N_k} \left( d_i - \left( 1 - \frac{F_{m,k}}{4 \pi k^{m+2}} \right) \right) |s_i|^2_{h^k_{(m)}} =0,
\end{equation}
where $d_1 , \dots , d_{N_k} \in \rl$, $m \in \nnb$, $h^k_{(m)} = \mathrm{FS}_k (H_{k,m})$, $H_{k,m} \in \CB_k$, is a hermitian metric on $L^k$ as constructed in \cite[Proposition 1 and Proof of Corollary 1]{yhextremal}, and $\{ s_i \}_{i=1}^{N_k}$ is an $H_{k,m}$-orthonormal basis. We need to conclude that there exists $C_5 >0$ such that
\begin{equation*}
	|d_i -1 | \le C_5 k^{c(n)-m}
\end{equation*}
for all $i=1 , \dots , N_k$, and for all large enough $k$, where $c(n)$ is some constant which depends only on $n$.

We apply the argument above. We first define two hermitian matrices $A:= \mathrm{diag} (d^{-1}_1 , \dots , d^{-1}_{N_k})$ and $B:=I$ with respect to the $H_{k,m}$-orthonormal basis $\{ s_i \}_{i=1}^{N_k}$. Note that they take a different form when they are represented with respect to an $H_k$-orthonormal basis, where $H_k = \mathrm{Hilb}(h^k)$ and $\omega_h$ is the extremal metric that is used as a reference metric in \cite{yhextremal}. We find
\begin{equation*}
	f_k (A,B;H_{k,m}) = \sum_{i=1}^{N_k} \left( d_i -  1 \right) |s_i|^2_{h^k_{(m)}},
\end{equation*}
and hence (\ref{eqextc}) can be re-written as
\begin{equation*}
	f_k (A,B; H_{k,m}) = - \frac{F_{m,k}}{4 \pi k^{m+2}}.
\end{equation*}
Taking $H_k = \mathrm{Hilb}(h^k)$ as above, we observe
\begin{equation*}
	f_k (A,B;H_k) = \frac{\mathrm{FS}_k (H_k)}{\mathrm{FS}_k (H_{k,m})} \left( \frac{\mathrm{FS}_k (H_{k,m})}{\mathrm{FS}_k (A)}  - \frac{\mathrm{FS}_k(H_{k,m})}{\mathrm{FS}_k (B)} \right) ,
\end{equation*}
which implies that
\begin{equation} \label{eqdfrfmc}
	f_k (A,B;H_k) = \frac{\mathrm{FS}_k (H_k)}{\mathrm{FS}_k (H_{k,m})} f_k (A,B;H_{k,m}) = - \frac{\mathrm{FS}_k (H_k)}{\mathrm{FS}_k (H_{k,m})} \frac{F_{m,k}}{4 \pi k^{m+2}}.
\end{equation}

We recall $\mathrm{FS}_k (H_k) = \bar{\rho}_k (\omega_h)^{-1}h^k$ \cite{rawnsley}, and
\begin{equation*}
	\mathrm{FS}_k (H_{k,m}) = e^{k \phi_{(m)}} h^k
\end{equation*}
in the terminology of \cite[Proposition 1]{yhextremal}; the only important point here is that $k \phi_{(m)}$ converges in $C^{\infty}$ as $k \to \infty$, and hence its derivatives are all uniformly bounded for all large enough $k$. Together with the expansion (\ref{eqbgexp}) of the Bergman function, we thus find that all derivatives of $\mathrm{FS}_k (H_k) / \mathrm{FS}_k (H_{k,m})$ are bounded uniformly for all large enough $k$, and hence there exists $C_6>0$ such that
\begin{equation*}
	\left\Vert \frac{\mathrm{FS}_k (H_k)}{\mathrm{FS}_k (H_{k,m})} \frac{F_{m,k}}{4 \pi k^{m+2}} \right\Vert_{W^{2,2} (\omega_h)} \le C_6k^{-m-2} \Vert F_{m,k} \Vert_{W^{2,2} (\omega_h)}
\end{equation*}
holds for all large enough $k$. As pointed out in \cite[page 2995, before the equation (20)]{yhextremal}, $\Vert F_{m,k} \Vert_{W^{2,2} (\omega_h)}$ is uniformly bounded for all large enough $k$. Thus there exists a constant $C_7 >0$ such that
\begin{equation} \label{eqesfkmm}
	\Vert f_k (A,B;H_k) \Vert_{W^{2,2} (\omega_h)} \le C_7 k^{-m-2}.
\end{equation}

We now recall that $A^{-1} = \mathrm{diag} (d_1 , \dots , d_{N_k})$ and $B^{-1}=I$ with respect to the $H_{k,m}$-orthonormal basis $\{ s_i \}_{i=1}^{N_k}$. When we represent $H_k$ as a matrix with respect to an $H_{k,m}$-orthonormal basis, we find that each entry of $H_{k}$ can be bounded uniformly for all large enough $k$ by the construction of $H_{k,m}$ as given in \cite[Proof of Corollary 1]{yhextremal} (see also \cite[Appendix]{donnum}), since it is defined as a $\mathrm{Hilb}(\tilde{h}^k)$ for some perturbation $\tilde{h}$ of $h$ such that $\log (\tilde{h}/h) = O(1/k)$. Thus there exists a constant $C_8 >0$ such that
\begin{equation*}
	\Vert A^{-1} - B^{-1} \Vert_{\mathrm{HS}(H_{m,k})} \le C_8 N_k \Vert A^{-1} - B^{-1} \Vert_{\mathrm{HS}(H_{k})},
\end{equation*}
holds for all large enough $k$. By recalling the asymptotic Riemann--Roch theorem $N_k = Vk^n + O(k^{n-1})$, we find that there exists a constant $C_9 >0$ such that
\begin{align*}
	|d_i-1| &\le \Vert A^{-1} - B^{-1} \Vert_{\mathrm{HS}(H_{m,k})} \\
	&\le C_9 k^n \Vert A^{-1} - B^{-1} \Vert_{\mathrm{HS}(H_{k})} \\
	&\le C_9 C_h k^{2n} \Vert f_k (A,B;H_k) \Vert_{W^{2,2} (\omega_h)} \\
	&\le C_7 C_9 C_h k^{2n-m-2}
\end{align*}
for all $i=1 , \dots , N_k$ and for all large enough $k$, by Theorem \ref{ppqifs} and (\ref{eqesfkmm}), as required.

\bibliography{FS.bib}

\end{document}